\documentclass[12pt]{article}

\usepackage[top=0.8in,bottom=1in,left=1in,right=1in]{geometry}
\usepackage{amsmath,amsthm,amssymb}
\usepackage{enumerate,xcolor,graphicx}
\usepackage{float}
\usepackage{authblk}

\usepackage{tikz}

\usepackage{hyperref}
\hypersetup{hidelinks}

\usetikzlibrary{arrows.meta,calc,positioning}

\newcommand{\Av}{\operatorname{Av}}

\DeclareMathOperator{\Var}{Var}
\DeclareMathOperator{\Bin}{Bin}
\newcommand{\dTV}{d_{\mathrm{TV}}}
\newcommand{\Law}{\mathcal{L}}

\newtheorem{thm}{Theorem}
\newtheorem{lem}[thm]{Lemma}
\newtheorem{prop}[thm]{Proposition}
\newtheorem{cor}[thm]{Corollary}
\newtheorem{openproblem}{Open Problem}

\theoremstyle{definition}
\newtheorem{defi}{Definition}
\newtheorem{rem}{Remark}
\newtheorem*{example}{Example}

\newcommand{\Z}{\mathbb{Z}}

\begin{document}

\title{The Distribution of Double Deficiencies in Pattern-Avoiding Permutations}

\author[1,4]{Tipaluck Krityakierne}
\author[2]{Thotsaporn Aek Thanatipanonda\thanks{Corresponding author. Email: \texttt{thotsaporn@gmail.com}}}
\author[3]{Doron Zeilberger}

\affil[1]{Department of Mathematics, Faculty of Science, Mahidol University,
272 Rama VI Rd., Ratchathewi, Bangkok 10400, Thailand}

\affil[2]{Science Division, Mahidol University International College,
999 Phutthamonthon 4 Rd., Salaya, Phutthamonthon,
Nakhon Pathom 73170, Thailand}

\affil[3]{Department of Mathematics, Rutgers University,
Hill Center--Busch Campus, 110 Frelinghuysen Rd.,
Piscataway, NJ 08854, USA}

\affil[4]{Centre of Excellence in Mathematics, MHESI,
272 Rama VI Rd., Ratchathewi, Bangkok 10400, Thailand}

\date{}

\maketitle

\vspace{-3em}

\begin{abstract}
We study the distribution of the number of double deficiencies (DD) in
permutations of length $n$ avoiding one or two patterns of length 3. 
Using structural decompositions of these avoidance classes---together with a lattice-path
decomposition in the \(321\)-avoiding case---we derive functional equations
and convolution-type recurrences that efficiently compute the corresponding
double-deficiency generating functions in all but one single-pattern case.
In the 321-avoiding permutations, the resulting generating function is algebraic;
we derive exact formulas for the mean and variance and prove that the
distribution is close in total variation to $\Bin(n-2,1/4)$, with an
explicit convergence rate.
We also identify a DD-preserving symmetry that yields DD-Wilf equivalences,
reducing the number of two-pattern cases that need to be considered separately.
For the resulting two-pattern classes, we obtain explicit recurrences, including
C-finite relations.
\end{abstract}

\section{Introduction}
\label{sec:intro}

Pattern avoidance is a classical topic in enumerative combinatorics.
Let
\[
[n]=\{1,2,\ldots,n\},
\]
and let $\mathfrak S_n$ denote the set of all permutations of $[n]$.
A permutation $\pi\in\mathfrak S_n$ is said to \emph{contain} a pattern
$\sigma\in\mathfrak S_k$ if some subsequence of $\pi$ is order-isomorphic
to $\sigma$; otherwise, $\pi$ is said to \emph{avoid} $\sigma$.
For a set of patterns $S$, let $\Av_n(S)$ denote the set of permutations
in $\mathfrak S_n$ avoiding every pattern in $S$.  When
$S=\{\sigma\}$, we write simply $\Av_n(\sigma)$.

For patterns of length $3$, the ordinary enumeration of these avoidance
classes is well understood. Simion and Schmidt systematically studied
permutations avoiding subsets of $\mathfrak S_3$, determining the
cardinalities of the resulting classes and giving several bijective
results \cite{SimionSchmidt1985}. In particular, for every
$\sigma\in\mathfrak S_3$, the single-pattern avoidance class
$\Av_n(\sigma)$ is counted by the Catalan number
\[
C_n=\frac{1}{n+1}\binom{2n}{n}.
\]
The simultaneous avoidance of two length-$3$ patterns has also been
studied extensively. For example, Vatter developed generating-tree
methods for enumerating permutations avoiding two distinct patterns
from $\mathfrak S_3$, possibly together with additional forbidden
patterns \cite{Vatter2002}.

A finer question is how permutation statistics are distributed within
these classes. Ordinary enumeration alone does not capture such
information: two avoidance classes may have the same cardinality but
different distributions of a given statistic. Several refinements of
length-$3$ pattern avoidance have therefore been studied. 
Robertson, Saracino, and Zeilberger refined the enumeration by the number of
fixed points---indices $k$ with $\pi(k)=k$---showing that the $132$-, $213$-,
and $321$-avoiding classes are equidistributed with respect to this statistic,
as are the $231$- and $312$-avoiding ones~\cite{RSZ}.
Elizalde determined the joint distribution of fixed points and excedances 
(indices with $\pi(k)>k$) for permutations
simultaneously avoiding subsets of $\mathfrak S_3$~\cite{Elizalde2004},
while Elizalde and Pak gave a bijection between $321$- and
$132$-avoiding permutations that preserves both
statistics~\cite{ElizaldePak2004}. Dokos, Dwyer, Johnson, Sagan, and
Selsor studied this phenomenon more generally through
\emph{statistic-Wilf equivalence}, and in particular determined the
equivalences arising from the inversion number and major index for sets
of forbidden patterns contained in $\mathfrak S_3$~\cite{DokosEtAl2012}.

In this paper, the refining statistic is the number of
\emph{double deficiencies}. For $\pi\in\mathfrak S_n$, an index
$k\in[n]$ is a double deficiency of $\pi$ if
\[
\pi(k)<k<\pi^{-1}(k),
\]
and we denote the number of double deficiencies of $\pi$ by
$\operatorname{DD}(\pi)$.
For example, consider the permutation $\pi=4251736$. Since
$\pi(6)=3<6<7=\pi^{-1}(6)$, the index $k=6$ is a double deficiency.
In fact, this is the only double deficiency of $\pi$, so
$\operatorname{DD}(\pi)=1$.

\begin{defi}
For a set of patterns $S$, define the double-deficiency generating
function by
\[
f_S(n,x)
=
\sum_{\pi\in\Av_n(S)}
x^{\operatorname{DD}(\pi)}.
\]
\end{defi}
For each fixed $n\ge2$, this is a polynomial in $x$ of degree at most $n-2$, since
neither $k=1$ nor $k=n$ can be a double deficiency, as $\pi(1)<1$ and
$\pi^{-1}(n)>n$ are impossible, so at most $n-2$ of the $n$ indices qualify.

For unrestricted permutations (i.e., when $S=\varnothing$),
inversion transforms double deficiencies into double excedances.
Under a version of Foata's fundamental transformation, double excedances
correspond to occurrences of the consecutive pattern $123$, which, by
reversal or complement symmetry, are equidistributed with occurrences
of the consecutive pattern $321$, or equivalently, with proper double
descents \cite{ElizaldeCF}.
The study of double descents and related
local permutation statistics goes back to Carlitz and Scoville
\cite{CarlitzScoville}. Elizalde and Noy obtained the full generating
function for the distribution of the number of occurrences of the
consecutive pattern $123$, and hence, by symmetry, $321$
\cite[Theorem~4.1]{ElizaldeNoy}. Consequently, their result also gives the
unrestricted double-deficiency distribution:
\[
\sum_{n\geq 0} f_{\varnothing}(n,x)\frac{z^n}{n!}
=
\frac{
2r\exp\!\left(\frac{1-x+r}{2}z\right)
}{
1+x+r-(1+x-r)\exp(rz)
},
\qquad
r=\sqrt{(x-1)(x+3)}.
\]
Fu later gave a context-free-grammar treatment and further refinements
involving proper double descents \cite{Fu}.

For pattern-avoiding permutations, however, the distribution of double
deficiencies has received much less attention than other classical
permutation statistics such as fixed points, excedances, inversions,
and major index.
The main previous work on double deficiencies in pattern-avoiding
permutations concerns $321$-avoiding permutations.
Motivated by a
conjecture of Marczinzik from the representation theory of Nakayama
algebras, Rubey and Stump studied double deficiencies using the
Billey--Jockusch--Stanley bijection between Dyck paths and
$321$-avoiding permutations \cite{BJS,RS}. In particular, they showed
that $321$-avoiding permutations with no double deficiencies are counted
by the Motzkin numbers. More recently, Krityakierne, Thanatipanonda, and
Zeilberger gave a shorter direct proof of the same enumeration
\cite{MAZ}.

Despite the existing results for unrestricted permutations and for the
zero-double-deficiency case in $\Av(321)$, to the best of our knowledge there has
been no systematic study of the full double-deficiency distribution for avoidance
classes defined by one or two patterns of length 3. In Section~\ref{sec:1pattern},
we focus on single-pattern avoidance. Using structural decompositions of these
avoidance classes, together with a lattice-path decomposition in the 321-avoiding
case, we derive functional equations and convolution-type recurrences for
efficiently computing the associated generating functions in all but one
single-pattern case.
We also identify, in Section~\ref{sec:DD-Wilf}, a symmetry that preserves
the number of double deficiencies. A composition of the usual reverse,
complement, and inverse operations interchanges the patterns $132$ and
$213$ while fixing the other four patterns of length $3$. In the
framework of statistic-Wilf equivalence, this yields a family of
\emph{DD-Wilf equivalences} and reduces the number of two-pattern
avoidance classes that need to be investigated separately in
Section~\ref{sec:2pattern}.
All Maple programs accompanying this work are available at
\url{https://sites.math.rutgers.edu/~zeilberg/mamarim/mamarimhtml/DD3av.html}.


\section{Single-Pattern Avoidance}
\label{sec:1pattern}

We now study the double-deficiency generating functions for
single-pattern avoidance classes of length $3$. For convenience, write
\[
f_{\sigma}(n,x)=f_{\{\sigma\}}(n,x),
\qquad \sigma\in S_3.
\]

\subsection{The class $\Av(321)$} 

The class $\Av(321)$ does not admit the block decomposition used later for $\Av(132)$, $\Av(213)$, and $\Av(231)$. We therefore use a lattice-path bijection under which double deficiencies correspond to a single step type. Previous work on double deficiencies in $321$-avoiding permutations uses the Billey--Jockusch--Stanley bijection with Dyck paths \cite{BJS,RS}. For the full distribution, we use Elizalde's cycle-diagram formulation \cite{ElizaldeCF} of the Foata--Zeilberger correspondence \cite{FoataZeilberger1990}, which gives a bijection with bicolored Motzkin paths.

Recall that a Motzkin path of length \(n\) is a lattice path from
\((0,0)\) to \((n,0)\), with up-steps \(U=(1,1)\), level steps
\(L=(1,0)\), and down-steps \(D=(1,-1)\), that never passes below the
\(x\)-axis. The number of such paths is the \(n\)th Motzkin number
\(M_n\).

In Elizalde's correspondence \cite[Sections~3.1 and~4.2]{ElizaldeCF}, a \(321\)-avoiding permutation \(\pi\) is
mapped to a bicolored Motzkin path in which level steps above the
\(x\)-axis have two possible colors, while level steps on the
\(x\)-axis have only one. The level steps on the \(x\)-axis correspond
to fixed points of \(\pi\), while the two types of level step above the
\(x\)-axis correspond, respectively, to double excedances and double
deficiencies. We adopt the convention of leaving fixed-point and
double-excedance level steps uncolored and coloring double-deficiency
level steps red. This gives a bijection between \(\Av_n(321)\) and the
set of all such paths of length \(n\), under which the number of red
level steps equals \(\operatorname{DD}(\pi)\).

Let \(F_1(z,x)\) count these paths, with \(z\) marking length and \(x\)
marking red level steps. Let \(F_0(z,x)\) be the auxiliary series in
which level steps on the \(x\)-axis may also be red. Thus
\[
F_1(z,x)=\sum_{n\ge0}f_{321}(n,x)z^n.
\]

\begin{figure}[ht]
\centering

\begin{tikzpicture}[
    scale=0.9,
    line cap=round,
    line join=round,
    every node/.style={font=\small}
]


\node at (0.8,4.6) {Path (drawn for $F_1$)};
\node at (5.8,4.6) {Contribution to $F_1$ at $x=1$};
\node at (11.4,4.6) {Contribution to $F_0$  at $x=1$};


\begin{scope}[shift={(0,3.5)}]
\draw[very thick] (0,0)--(0.8,0.8)--(1.6,0)--(2.4,0);
\end{scope}

\node at (6.3,3.5) {$1$};
\node at (11.5,3.5) {$2$};


\begin{scope}[shift={(0,2.3)}]
\draw[very thick] (0,0)--(0.8,0);
\draw[very thick] (0.8,0)--(1.6,0.8)--(2.4,0);
\end{scope}

\node at (6.3,2.3) {$1$};
\node at (11.5,2.3) {$2$};


\begin{scope}[shift={(0,1.1)}]
\draw[very thick] (0,0)--(0.8,0.8);
\draw[very thick] (0.8,0.8)--(1.6,0.8);
\draw[red,very thick] (0.8,0.92)--(1.6,0.92);
\draw[very thick] (1.6,0.8)--(2.4,0);
\end{scope}

\node at (6.3,1.1) {$2$};
\node at (11.5,1.1) {$2$};


\begin{scope}[shift={(0,-0.1)}]
\draw[very thick] (0,0)--(2.4,0);
\fill (0,0) circle (1.5pt);
\fill (0.8,0) circle (1.5pt);
\fill (1.6,0) circle (1.5pt);
\fill (2.4,0) circle (1.5pt);
\end{scope}

\node at (6.3,-0.1) {$1$};
\node at (11.5,-0.1) {$2^3=8$};




\end{tikzpicture}
\caption{The four Motzkin path shapes of length \(3\), drawn for \(F_1\):
level steps on the \(x\)-axis must be uncolored, while those above it may
be either uncolored or red, as the parallel segments indicate. In \(F_0\),
every level step may be either.}

\label{fig:motzkin3}
\end{figure}

At \(x=1\), Figure~\ref{fig:motzkin3} gives
\[
[z^3]F_1(z,1)=5=C_3,
\qquad
[z^3]F_0(z,1)=14=C_4.
\]
More generally,
\begin{equation}\label{eq:catalan-claim}
[z^n]F_1(z,1)=C_n,
\qquad
[z^n]F_0(z,1)=C_{n+1}.
\end{equation}
The first identity is the Catalan enumeration of \(\Av_n(321)\); both
identities are verified in Remark~\ref{rem:catalan}. A first-return
decomposition of an \(F_0\)-path gives
\begin{equation}\label{eq:F0-func}
F_0(z,x)
=
1+(1+x)zF_0(z,x)+z^2F_0(z,x)^2.
\end{equation}

\begin{figure}[ht]
\begin{center}
\begin{tikzpicture}[
    scale=0.9,
    line cap=round,
    line join=round,
    every node/.style={font=\small}
]


\node[anchor=east] at (0,1.4) {Level step};

\draw[very thick] (0.5,1.4)--(1.5,1.4);

\node at (2.2,1.4) {$+$};

\node[draw,rounded corners] at (3.4,1.4) {$F_1$};

\node at (5.8,1.4)
{$\Longrightarrow\quad zF_1$};


\node[anchor=east] at (0,-0.2) {First return};

\draw[very thick] (0.5,-0.2)--(1.2,0.5);
\draw[dashed] (1.2,0.5)--(2.7,0.5);
\draw[very thick] (2.7,0.5)--(3.4,-0.2);

\node at (1.95,0.8) {$F_0$};

\node at (4.0,-0.2) {$+$};

\node[draw,rounded corners] at (5.2,-0.2) {$F_1$};

\node at (8.0,-0.2)
{$\Longrightarrow\quad z^2F_0F_1$};

\end{tikzpicture}
\caption{First-return decomposition for \(F_1(z,x)\). The analogous
decomposition for \(F_0(z,x)\) is obtained by replacing the boxes
labeled \(F_1\) by \(F_0\) and allowing the initial level step to be
either uncolored or red. The two nonempty contributions then become
\((1+x)zF_0\) and \(z^2F_0^2\).}
\label{fig:F1-decomposition}
\end{center}
\end{figure}

For \(F_1\), the initial level step has weight \(z\), while the portion
enclosed by the first return is an \(F_0\)-path (Figure~\ref{fig:F1-decomposition}). Hence,
\begin{equation}\label{eq:F1-func}
F_1(z,x)
=
1+zF_1(z,x)+z^2F_0(z,x)F_1(z,x).
\end{equation}

Solving the quadratic equation for \(F_0(z,x)\) gives
\[
F_0(z,x)
=
\frac{1-(1+x)z-\sqrt{1-2z-2xz-3z^2+2xz^2+x^2z^2}}
{2z^2}.
\]

Substituting into the equation for \(F_1(z,x)\) and rearranging yields

\[
F_1(z,x)
=
\frac{2}
{1-z+xz+\sqrt{1-2z-2xz-3z^2+2xz^2+x^2z^2}}.
\]

This is the desired bivariate generating function over \(\Av(321)\).

\begin{rem}\label{rem:catalan}
At \(x=1\), the discriminant is \(1-4z\), and
\[
F_0(z,1)
=
\frac{1-2z-\sqrt{1-4z}}{2z^2},
\qquad
F_1(z,1)
=
\frac{2}{1+\sqrt{1-4z}}
=
\frac{1-\sqrt{1-4z}}{2z}
=
C(z),
\]
where \(C(z)=\sum_{n\ge0}C_nz^n\). Since \(C=1+zC^{2}\),
\[
F_0(z,1)
=
\frac{1-2z-\sqrt{1-4z}}{2z^{2}}
=
\frac{C(z)-1}{z}
=
C(z)^{2}.
\]
This proves~\eqref{eq:catalan-claim}.
\end{rem}


\paragraph{Coefficient recurrences.}
Equating coefficients in~\eqref{eq:F0-func}--\eqref{eq:F1-func}, write
\[
F_0(z,x)
=
\sum_{n\ge0}g(n,x)z^n,
\qquad
F_1(z,x)
=
\sum_{n\ge0}f_{321}(n,x)z^n.
\]
Then \(g(0,x)=f_{321}(0,x)=1\), and for \(n\ge1\),
\[
g(n,x)
=
(1+x)g(n-1,x)
+
\sum_{i=2}^{n}
g(i-2,x)\,g(n-i,x).
\]

\[
f_{321}(n,x)
=
f_{321}(n-1,x)
+
\sum_{i=2}^{n}
g(i-2,x)\,
f_{321}(n-i,x).
\]

\begin{figure}[ht]
\centering
\includegraphics[width=0.75\textwidth]{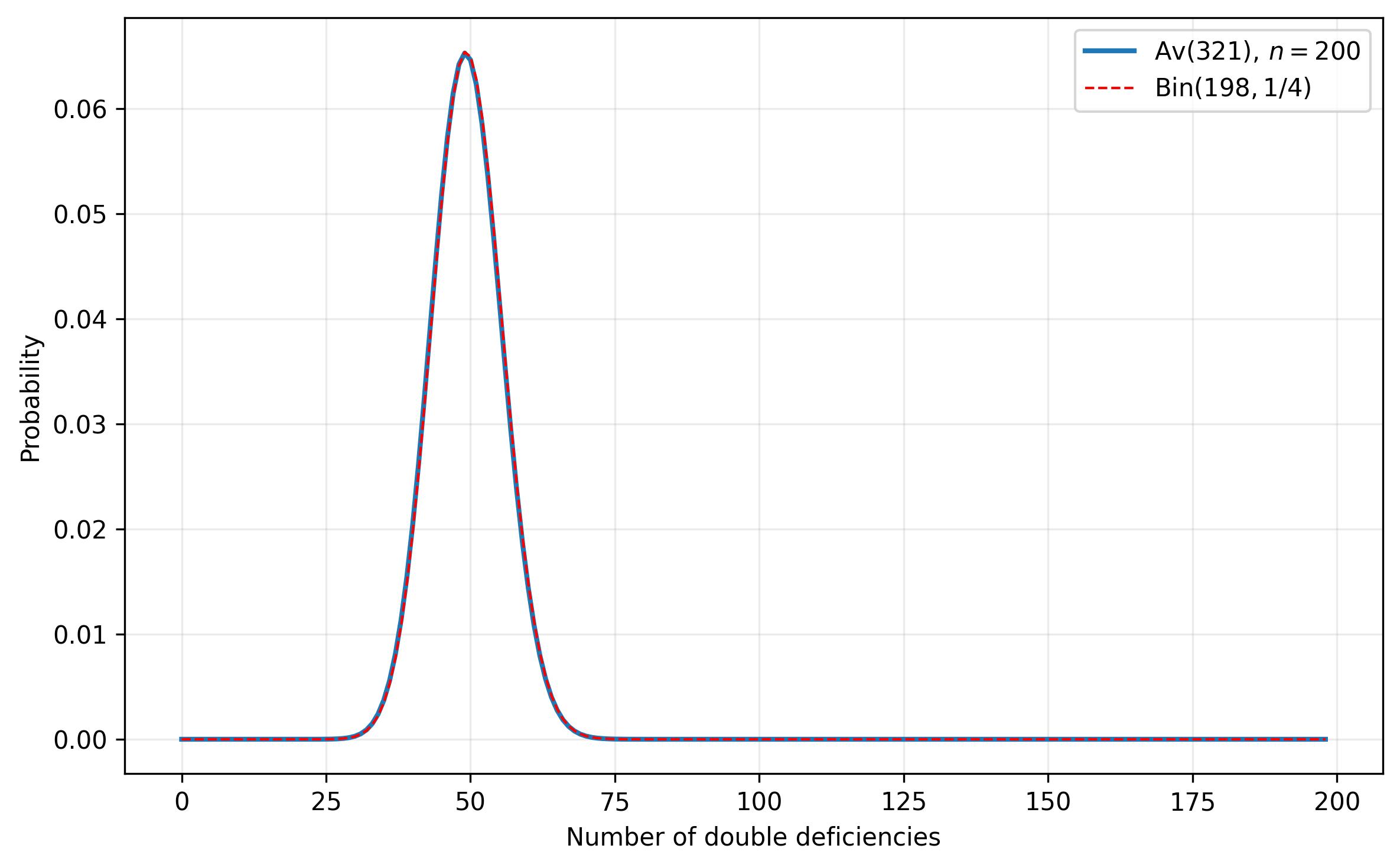}
\caption{Probability distribution of the number of double deficiencies
in \(\Av_{200}(321)\), together with the binomial distribution
\(\operatorname{Bin}(198,\tfrac14)\).}
\label{fig:321distribution}
\end{figure}

Figure~\ref{fig:321distribution}, computed from this recurrence, suggests
the binomial approximation proved next.

\paragraph{A binomial approximation.}

For probability distributions \(\mu\) and \(\nu\) on \(\mathbb Z\), write
\[
d_{\mathrm{TV}}(\mu,\nu)
=
\frac12\sum_{j\in\mathbb Z}
\left|\mu\{j\}-\nu\{j\}\right|.
\]

\begin{prop}\label{prop:binomial}
Let \(X_n\) denote the number of double deficiencies in a uniformly
random permutation in \(\Av_n(321)\). For \(n\ge2\), the support of
\(X_n\) is \(\{0,1,\ldots,n-2\}\), and
\[
\mathbb E[X_n]
=
\frac{(n-1)(n-2)}{2(2n-1)}
\]
and
\[
\operatorname{Var}(X_n)
=
\frac{(n+1)(n-2)(3n^2-8n+3)}
     {2(2n-1)^2(2n-3)}.
\]
Moreover,
\[
d_{\mathrm{TV}}
\left(
\mathcal L(X_n),
\operatorname{Bin}\left(n-2,\frac14\right)
\right)
=
O\left(n^{-1/2}(\log n)^{1/4}\right).
\]
In particular, $X_n$ is asymptotically normal.
\end{prop}

The moment formulas follow by differentiating the functional equations and
extracting coefficients, while the binomial approximation follows from
singularity analysis applied to \(F_1(z,x)\); both are proved in
Appendix~\ref{app:prop1}.


\paragraph{A holonomic recurrence.}
Eliminating \(g(n,x)\) yields the following third-order recurrence, valid for
\(n\ge0\) as an identity in \(\mathbb{Z}[x]\):
\[
\begin{aligned}
(n+4)x\,f_{321}(n+3,x)
={}&
\Bigl((5x^2+9x-4)+(2x^2+3x-1)n\Bigr)
f_{321}(n+2,x)
\\
&
-(x-1)
\Bigl((x^2+8x+5)+(x^2+5x+2)n\Bigr)
f_{321}(n+1,x)
\\
&
+(x-1)^2(x+3)(n+1)
f_{321}(n,x),
\end{aligned}
\]
with initial conditions
\[
f_{321}(0,x)=f_{321}(1,x)=1,
\qquad
f_{321}(2,x)=2.
\]
At \(x=0\) the leading coefficient vanishes, and the identity reduces to
\[
(n+4)M_{n+2}=(2n+5)M_{n+1}+3(n+1)M_n,
\]
the classical recurrence for the Motzkin numbers, consistent with
\(f_{321}(n,0)=M_n\).


\subsection{The class $\Av(312)$}

\begin{prop}
\label{prop:312}
If \(k\) is a double deficiency of \(\pi\), then \(\pi\) contains an occurrence of the pattern \(312\).
\end{prop}

\begin{proof}
Suppose \(k\) is a double deficiency of \(\pi\). Then
\(
\pi(k)<k<\pi^{-1}(k).
\)
Among the first \(k-1\) positions, neither the value \(\pi(k)\) nor the
value \(k\) can occur. Since there are \(k-1\) positions but only
\(k-2\) remaining values from \(\{1,\ldots,k\}\), the pigeonhole principle
implies that some position \(i<k\) satisfies
\(
\pi(i)>k.
\)
Hence,
\(
\pi(i)>k>\pi(k),
\)
so the entries
\(
\pi(i),\ \pi(k),\ k
\)
at positions \(i\), \(k\) and \(\pi^{-1}(k)\), respectively, form an
occurrence of the pattern \(312\).
\end{proof}

\begin{cor}
For every \(n\), every permutation in \(\operatorname{Av}_n(312)\) has no double deficiencies. Consequently,
\[
f_{312}(n,x)=C_n,
\]
where \(C_n\) denotes the \(n\)th Catalan number.
\end{cor}

\begin{rem}
The converse of the proposition is false. For example, the permutation
\(3412\) contains the pattern \(312\) but has no double deficiency.
\end{rem}

\subsection{The class \(\Av(132)\)}
\label{sec:132}

Let \(\pi\in\Av_n(132)\), with its maximum in position \(i\), and write
\[
\pi=L\,n\,R.
\]
Avoidance of \(132\) forces every entry of \(L\) to exceed every entry
of \(R\); otherwise \(a,n,b\), with \(a\in L\), \(b\in R\), and
\(a<b\), is a \(132\)-pattern. Thus, as sets of values,
\[
L=\{n-i+1,\ldots,n-1\},
\qquad
R=\{1,\ldots,n-i\}.
\]
After standardization, both blocks remain \(132\)-avoiding. The statistic
is not preserved, however, because positions and values may be shifted by
different amounts; the following parameter records their difference.

\paragraph{Shifted double deficiencies.}

Suppose a block of length \(m\) occupies positions \(p+1,\ldots,p+m\),
has value set \(\{q+1,\ldots,q+m\}\), and standardizes to
\(\beta\in\mathfrak S_m\). Then
\[
\pi(p+j)=q+\beta(j),
\qquad 1\le j\le m.
\]

At local position \(j\), the global position is \(k=p+j\), while the
global value \(k\) becomes
\[
k-q=p+j-q=j+c,
\qquad c:=p-q.
\]
When \(1\le j+c\le m\),
\[
\pi(k)=q+\beta(j)
\qquad\text{and}\qquad
\pi^{-1}(k)=p+\beta^{-1}(j+c).
\]
Consequently,
\[
\pi(k)<k<\pi^{-1}(k)
\quad\Longleftrightarrow\quad
\beta(j)<j+c<\beta^{-1}(j+c)+c.
\]
This motivates the following definition.

\begin{defi}
\label{def:DD_c}
Let \(\pi\in \mathfrak S_n\) and \(c\in\mathbb Z\). The number of
\emph{shifted double deficiencies} of \(\pi\) with shift \(c\) is
\[
\operatorname{DD}_c(\pi)
=
\#\left\{
j\in\{1,\ldots,n\}:
1\le j+c\le n,\;
\pi(j)<j+c<\pi^{-1}(j+c)+c
\right\}.
\]
\end{defi}

Thus \(\operatorname{DD}_0(\pi)=\operatorname{DD}(\pi)\).

\begin{example}
Let
\(
\pi=31245\in\Av_5(132)
\text{ and }
c=1.
\)
Since \(\pi^{-1}=23145\), the shifted condition holds at precisely
\(j=3,4\):
\[
\begin{aligned}
\pi(3)=2 &< 3+1 < \pi^{-1}(4)+1=5,\\
\pi(4)=4 &< 4+1 < \pi^{-1}(5)+1=6.
\end{aligned}
\]
Thus \(\operatorname{DD}_1(\pi)=2\), whereas
\(
\pi(2)=1<2<\pi^{-1}(2)=3.
\)
Thus \(\operatorname{DD}(\pi)=1\).
\end{example}

\paragraph{Generating function.}

For integers \(n\ge0\) and \(c\in\mathbb Z\), define
\[
f(n,c,x)
=
\sum_{\pi\in\Av_n(132)}
x^{\operatorname{DD}_c(\pi)}.
\]
The ordinary generating function is
\(
f_{132}(n,x)=f(n,0,x).
\)

If the indices of \(\pi\) are counted by \(\operatorname{DD}_c\), then those of a
standardized block with offsets \(p,q\) are counted by
\(\operatorname{DD}_{c+p-q}\); that is, the block carries shift \(c+p-q\).
For \(\pi=L\,n\,R\), the offset differences are \(-(n-i)\) for \(L\) and \(i\)
for \(R\), so the left and right blocks contribute
\(f(i-1,c-(n-i),x)\) and \(f(n-i,c+i,x)\), respectively.

\begin{lem}
Let \(\pi=L\,n\,R\in\Av_n(132)\) with \(\pi(i)=n\), and let \(c\in\Z\).
The two recursive factors account for every shifted double deficiency
of \(\pi\), except possibly the position \(j=n-c\), whose shifted
value is \(j+c=n\). This position is a shifted double deficiency if and only if
\[
n\in\{c+1,\ldots,c+i-1\}.
\]
\end{lem}

\begin{proof}
The position \(i\) cannot contribute because \(\pi(i)=n\). If a position
lies in \(L\) while its shifted value lies in \(R\), the first inequality
fails; in the reverse situation, the shifted value occurs to the left and
the second inequality fails. The only remaining shifted value is \(n\),
which forces \(j=n-c\). Since \(\pi^{-1}(n)=i\), its two inequalities are
\(\pi(n-c)<n<i+c\). They hold exactly when \(n-c\) lies in \(L\), namely
when \(1\le n-c\le i-1\): membership in \(L\) gives the first inequality,
and \(n-c<i\) gives the second. Equivalently,
\(n\in\{c+1,\ldots,c+i-1\}\).
\end{proof}

\begin{prop}
\label{prop:132}
For every $n\ge1$ and $c\in\mathbb Z$,
\[
f(n,c,x)
=
\sum_{i=1}^{n}
x^{\Delta(n,c,i)}
f\bigl(i-1,c-(n-i),x\bigr)
f(n-i,c+i,x),
\]
where
\[
\Delta(n,c,i)
=
\begin{cases}
1,
&
n\in\{c+1,\ldots,c+i-1\},
\\
0,
&
\text{otherwise},
\end{cases}
\]
with initial condition
\[
f(0,c,x)=1.
\]
In particular,
\[
f_{132}(n,x)=f(n,0,x).
\]
\end{prop}

Figure~\ref{fig:132dist} shows the distribution obtained from the
recurrence at \(n=140\).

\begin{figure}
\centering
\includegraphics[width=0.75\textwidth]{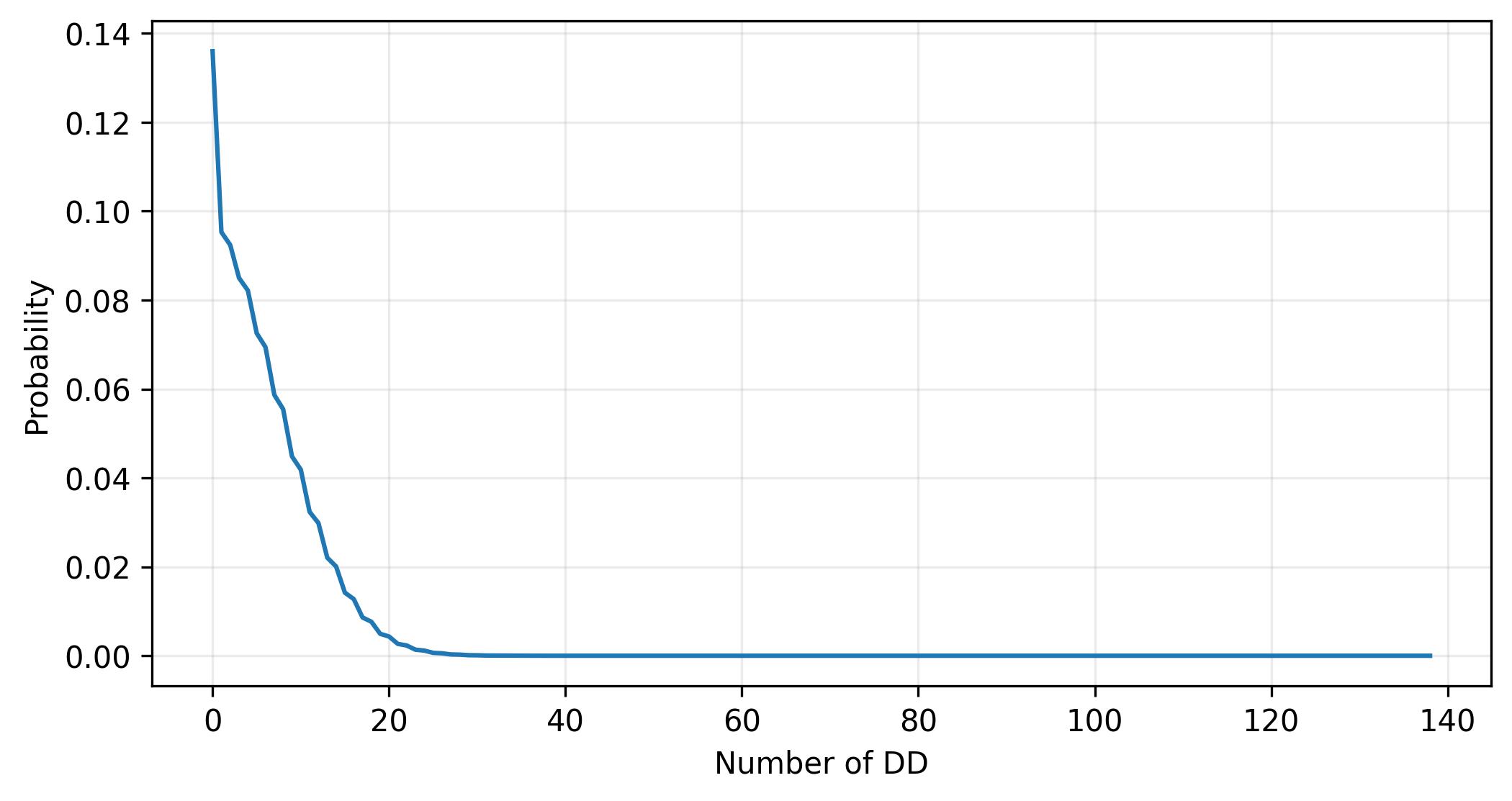}
\caption{Probability distribution of the number of double deficiencies
in \(\Av_{140}(132)\).}
\label{fig:132dist}
\end{figure}

The distribution is strongly asymmetric and concentrated near small
values; Figure~\ref{fig:132mean} shows the corresponding mean.

\begin{figure}
\centering
\includegraphics[width=0.75\textwidth]{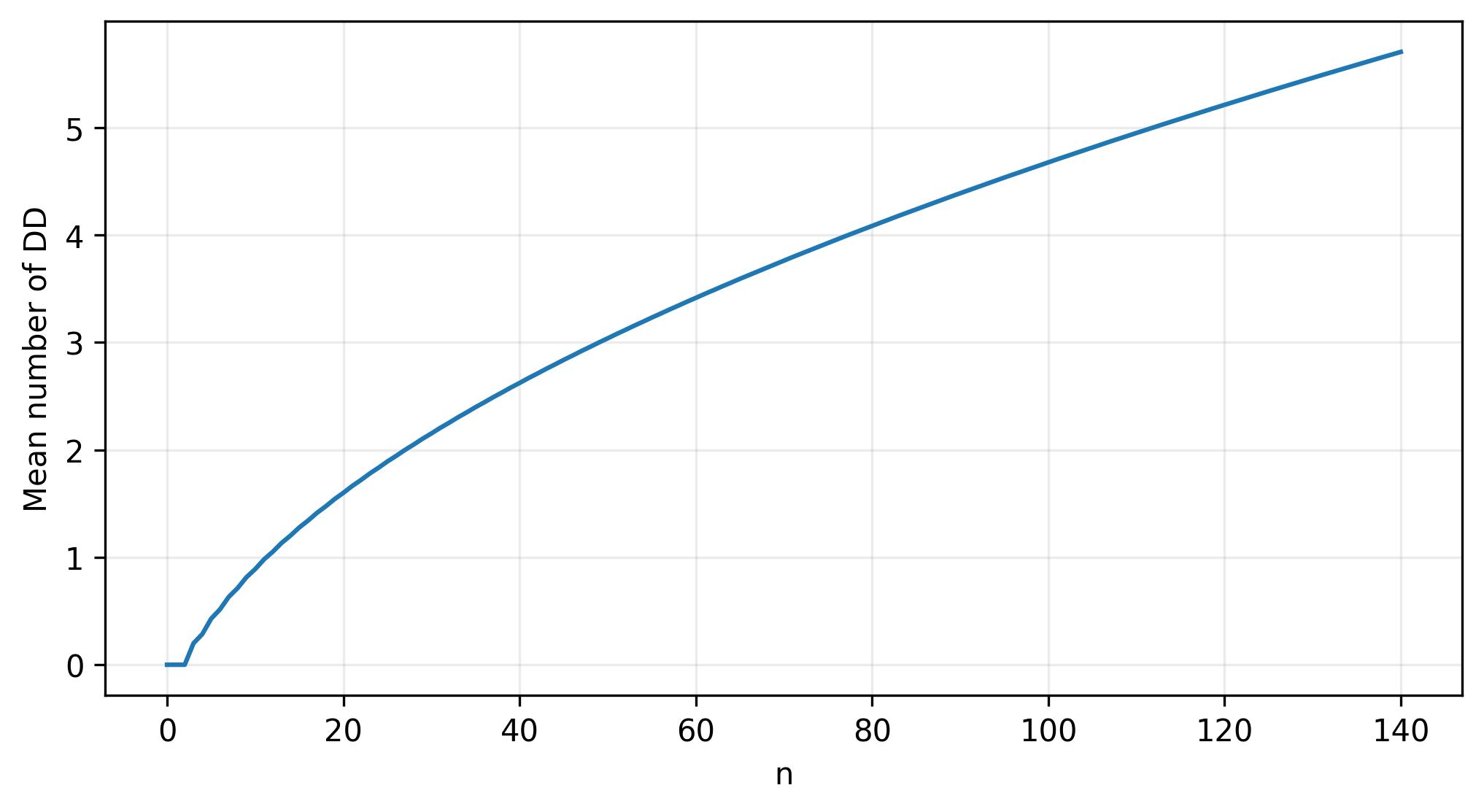}
\caption{Mean number of double deficiencies in \(132\)-avoiding
permutations of length \(n\), for \(0\le n\le140\).}
\label{fig:132mean}
\end{figure}


\subsection{The class $\Av(213)$}
\label{sec:213}

For this subsection, let
\[
f(n,c,x)
=
\sum_{\pi\in\Av_n(213)}
x^{\operatorname{DD}_c(\pi)}.
\]

If \(\pi\in\Av_n(213)\) has its minimum in position \(i\), write
\[
\pi=L\,1\,R.
\]
Avoidance of \(213\) gives, as sets of values, 
\[
L=\{n-i+2,\ldots,n\},
\qquad
R=\{2,\ldots,n-i+1\}.
\]
The standardized blocks remain \(213\)-avoiding and have shifts
\(c-(n-i+1)\) and \(c+i-1\). Cross-block contributions fail as in the
$\Av(132)$  case. The separating position \(i\) contributes exactly when
its shifted value \(i+c\) lies in \(R\), namely when
\(i+c\in\{2,\ldots,n-i+1\}\).

\begin{prop}
For every $n\ge1$ and $c\in\mathbb Z$,
\[
f(n,c,x)
=
\sum_{i=1}^{n}
x^{\Delta(n,c,i)}
f\bigl(i-1,c-(n-i+1),x\bigr)
f\bigl(n-i,c+i-1,x\bigr),
\]
where
\[
\Delta(n,c,i)
=
\begin{cases}
1,
&
i+c\in\{2,\ldots,n-i+1\},
\\
0,
&
\text{otherwise},
\end{cases}
\]
with initial condition
\[
f(0,c,x)=1.
\]
In particular,
\[
f_{213}(n,x)=f(n,0,x).
\]
\end{prop}

\begin{proof}
The block shifts count all contributions internal to \(L\) and \(R\),
and the cross-block argument from the $\Av(132)$  case rules out the other
crossings. At the separating position, occupied by \(1\), the condition is
\[
1<i+c<\pi^{-1}(i+c)+c.
\]
It holds exactly when \(i+c\in R=\{2,\ldots,n-i+1\}\): then
\(1<i+c\) and \(\pi^{-1}(i+c)>i\), giving the two strict inequalities.
Thus the factor \(x^{\Delta(n,c,i)}\) records the only possible extra
contribution, and summing over \(i\) gives the formula.
\end{proof}

The DD-preserving bijection described in Section~\ref{sec:DD-Wilf}
also shows that
\[
f_{213}(n,x)=f_{132}(n,x).
\]

\subsection{The class $\Av(231)$}
\label{sec:231}

For \(n\ge0\) and \(c\ge0\), set
\[
f(n,c,x)
=
\sum_{\pi\in\Av_n(231)}
x^{\operatorname{DD}_c(\pi)}.
\]
If the maximum of
\(\pi\in\Av_n(231)\) is in position \(i\), write
\[
\pi=L\,n\,R.
\]
Avoidance of \(231\) forces every entry of \(L\) below every entry of
\(R\), and hence, as sets of values,
\[
L=\{1,\ldots,i-1\},
\qquad
R=\{i,\ldots,n-1\}.
\]

The two standardized blocks contribute \(f(i-1,c,x)\) and
\(f(n-i,c+1,x)\). Their interaction is described next.

\begin{lem}
\label{lem:231}
Let \(\pi=L\,n\,R\in\Av_n(231)\), where \(\pi(i)=n\), and suppose that
\(0\le c<n\). The shifted double deficiencies not counted internally by
the two recursive factors are precisely the positions \(j\) satisfying
\[
1\le j\le i-1
\qquad\text{and}\qquad
i\le j+c\le n.
\]
Their number is
\[
\Delta(n,c,i)
=
\min\{c,\ i-1,\ n-c,\ n-i+1\}.
\]
\end{lem}

\begin{proof}
No new contribution can have its position in \(R\): if \(j+c<n\) it
is internal to \(R\), while if \(j+c=n\) the value occurs at position
\(i<j\). The position \(i\) also fails because \(\pi(i)=n\). Thus only a
position in \(L\) whose shifted value lies in \(R\cup\{n\}\) can give a new contribution,
which is equivalent to
\[
1\le j\le i-1
\qquad\text{and}\qquad
i\le j+c\le n.
\]
For such \(j\), \(\pi(j)\le i-1<j+c\), and the value \(j+c\) occurs to
the right of \(j\); hence both inequalities hold. Therefore,
\[
\Delta(n,c,i)
=
\#\bigl([1,i-1]\cap[i-c,n-c]\bigr).
\]
The intersection has lower endpoint \(\max(1,i-c)\) and upper endpoint
\(\min(i-1,n-c)\). Comparing the two possible upper endpoints with the
two possible lower endpoints gives
\[
\Delta(n,c,i)
=\min\{c,i-1,n-c,n-i+1\}.
\]

\end{proof}

\begin{prop}
For \(n\ge1\) and \(0\le c<n\),
\[
f(n,c,x)
=
\sum_{i=1}^{n}
x^{\Delta(n,c,i)}
f(i-1,c,x)
f(n-i,c+1,x),
\]
where
\[
\Delta(n,c,i)
=
\min\{c,\ i-1,\ n-c,\ n-i+1\}.
\]
The initial and boundary conditions are
\[
f(0,c,x)=1
\]
and
\[
f(n,c,x)=C_n,
\qquad c\ge n,
\]
where \(C_n\) is the \(n\)th Catalan number. In particular,
\[
f_{231}(n,x)=f(n,0,x).
\]
\end{prop}

\begin{proof}
Sum the decomposition in the lemma over \(i\). The boundary condition
follows because \(c\ge n\) leaves no admissible shifted value, so every
one of the \(|\Av_n(231)|=C_n\) permutations has weight \(1\).
\end{proof}

Figure~\ref{fig:231dist} shows the distribution computed at \(n=120\).

\begin{figure}[ht]
\centering
\includegraphics[width=0.75\textwidth]{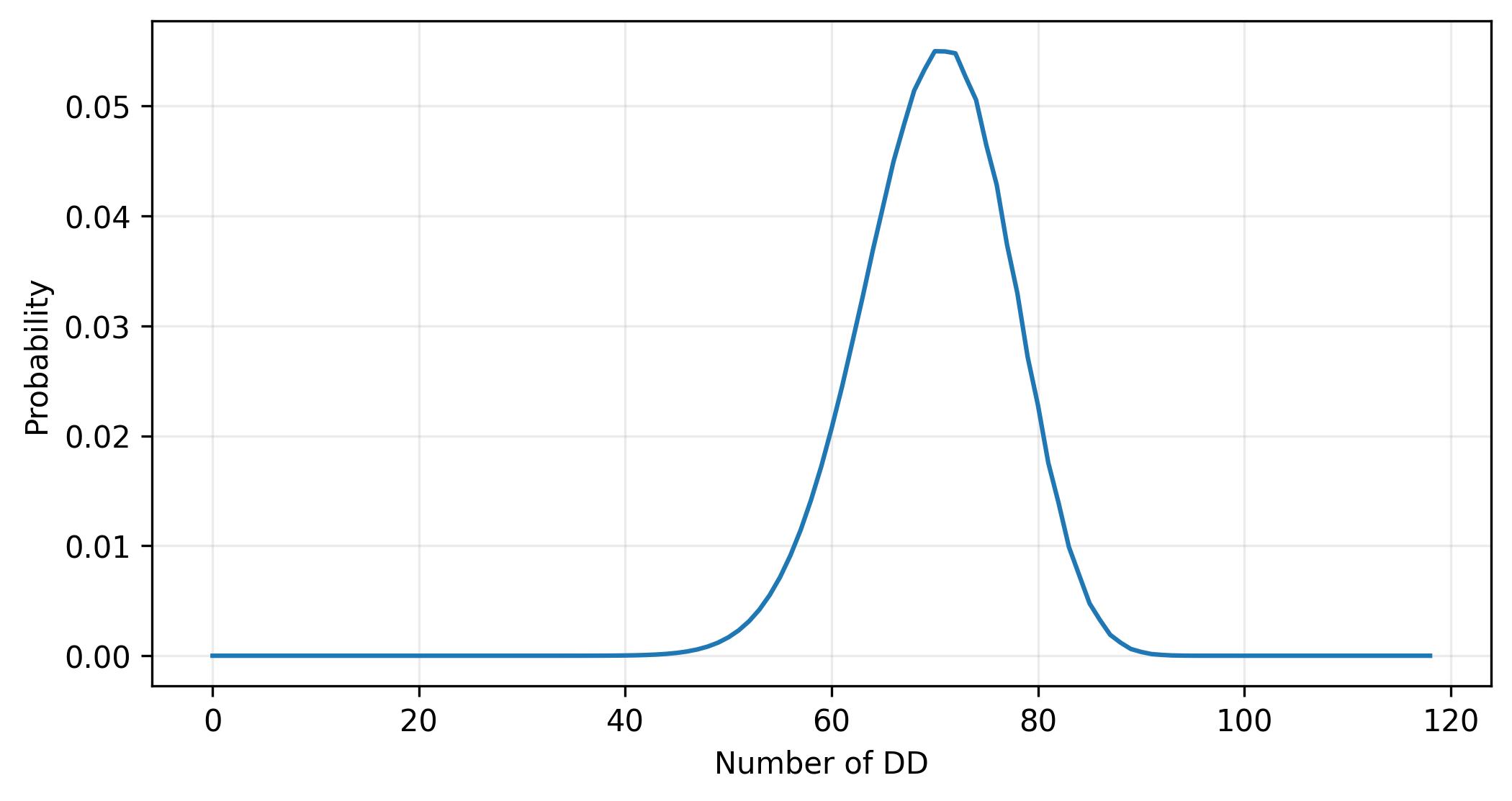}
\caption{Probability distribution of the number of double deficiencies
in \(\Av_{120}(231)\).}
\label{fig:231dist}
\end{figure}

Empirically, \(\Av(231)\) has the largest mean among the six
single-pattern classes; unlike the other recurrences, its interaction
term may add several shifted deficiencies at once. Figure~\ref{fig:231mean}
shows this mean as a function of \(n\).

\begin{figure}[h]
\centering
\includegraphics[width=0.75\textwidth]{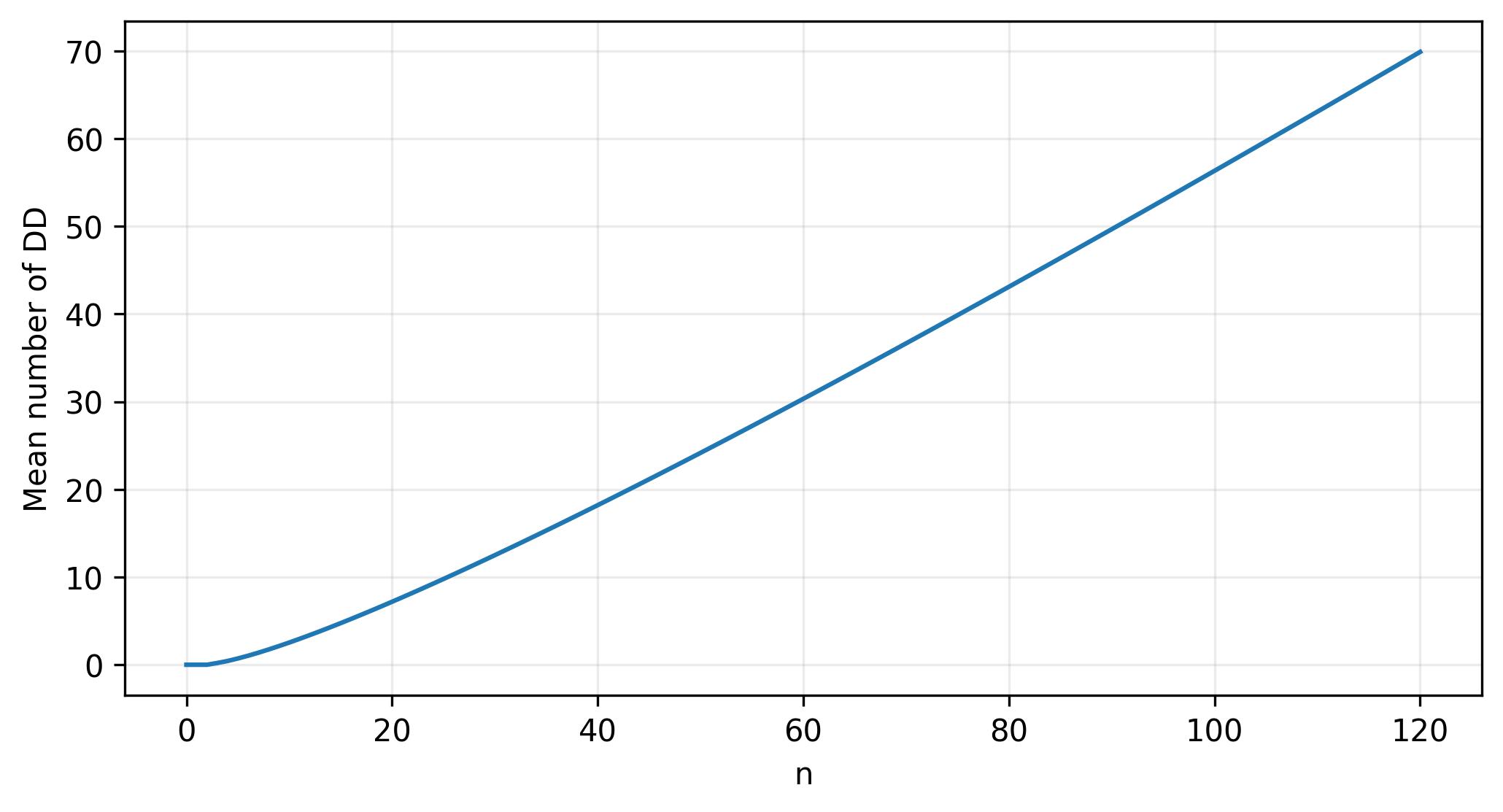}
\caption{Mean number of double deficiencies in \(231\)-avoiding
permutations of length \(n\), for \(0\le n\le120\).}
\label{fig:231mean}
\end{figure}


\subsection{The class $\Av(123)$}

The class \(\Av(123)\) does not appear to admit a decomposition analogous to those used for \(\Av(132)\), \(\Av(213)\), or \(\Av(231)\).
A path-based approach analogous in spirit to that used for \(\Av(321)\) may be possible. However, the relevant indices do not align with the double-deficiency condition, and we have not obtained a closed recurrence for \(f_{123}(n,x)\).

\begin{openproblem}
\label{op:123}
Find an efficient method for computing the generating functions
\[
f_{123}(n,x)
=
\sum_{\pi\in\Av_n(123)}
x^{\operatorname{DD}(\pi)}.
\]
\end{openproblem}

\section{DD-Wilf equivalences}
\label{sec:DD-Wilf}

For a set of patterns $S$, recall that
\[
f_S(n,x)
=
\sum_{\pi\in\Av_n(S)}
x^{\operatorname{DD}(\pi)}
\]
is the generating function for the distribution of double deficiencies over
$\Av_n(S)$. 
Following the standard terminology of statistic-Wilf
equivalence \cite{DokosEtAl2012}, we say that two sets of patterns $S$
and $T$ are \emph{DD-Wilf equivalent} if
\[
f_S(n,x)=f_T(n,x)
\]
for every $n\geq 0$.

To establish several DD-Wilf equivalences, we use the following standard transformations of permutations \cite{SimionSchmidt1985}. For $\pi\in\mathfrak S_n$, let $R(\pi)$, $C(\pi)$, and $I(\pi)$ denote its reverse, complement, and inverse, respectively, defined by
\[
R(\pi)(i)=\pi(n+1-i),\qquad
C(\pi)(i)=n+1-\pi(i),\qquad
I(\pi)=\pi^{-1}.
\]

We consider the composition
\[
\Phi=C\circ R\circ I.
\]
Explicitly,
\[
\Phi(\pi)(i)
=
n+1-\pi^{-1}(n+1-i).
\]

For the proof below, let
\[
\mathcal{D}(\pi)
=
\left\{
k\in[n]:
\pi(k)<k<\pi^{-1}(k)
\right\}
\]
denote the set of double-deficiency indices of $\pi$.  Thus,
\(
\operatorname{DD}(\pi)=|\mathcal{D}(\pi)|.
\)

\begin{prop}
For every $\pi\in\mathfrak S_n$,
\[
\operatorname{DD}(\Phi(\pi))
=
\operatorname{DD}(\pi).
\]
More precisely,
\[
\mathcal{D}(\Phi(\pi))
=
\{\,n+1-k:k\in \mathcal{D}(\pi)\,\}.
\]
\end{prop}

\begin{proof}
Let
\(
\rho=\Phi(\pi)=C(R(I(\pi))).
\)
Then
\(
\rho(i)
=
n+1-\pi^{-1}(n+1-i).
\)
Since
\(
\rho^{-1}(i)
=
n+1-\pi(n+1-i),
\)
the index $i$ is a double deficiency of $\rho$ precisely when
\[
n+1-\pi^{-1}(n+1-i)
<
i
<
n+1-\pi(n+1-i).
\]
Set
\(
k=n+1-i.
\)
The preceding inequalities are equivalent to
\(
\pi(k)<k<\pi^{-1}(k).
\)
Thus,
\[
i\in \mathcal{D}(\rho)
\quad\Longleftrightarrow\quad
n+1-i\in \mathcal{D}(\pi).
\]
Hence,
\(
\mathcal{D}(\Phi(\pi))
=
\{\,n+1-k:k\in \mathcal{D}(\pi)\,\},
\)
and therefore
\(
\operatorname{DD}(\Phi(\pi))
=
\operatorname{DD}(\pi).
\)
\end{proof}

We now determine how $\Phi$ acts on patterns of length $3$. Let
$\sigma\in\mathfrak S_3$ and $\pi\in\mathfrak S_n$. Since the reverse,
complement, and inverse operations are invertible and carry pattern
occurrences to pattern occurrences,
\[
\pi\text{ contains }\sigma
\quad\Longleftrightarrow\quad
\Phi(\pi)\text{ contains }\Phi(\sigma).
\]

For $S\subseteq\mathfrak S_3$, define
\[
\Phi(S)=\{\Phi(\sigma):\sigma\in S\}.
\]
It follows that $\Phi$ restricts to a bijection
\[
\Phi:\Av_n(S)\longrightarrow\Av_n(\Phi(S)).
\]

On the six patterns of length $3$, $\Phi$ acts as follows:
\[
\begin{array}{c|cccccc}
\sigma       &123&321&132&213&231&312\\ \hline
\Phi(\sigma) &123&321&213&132&231&312
\end{array}
\]
Thus $\Phi$ fixes $123$, $321$, $231$, and $312$, and interchanges $132$ and $213$.

\begin{cor}
For every
\[
S\subseteq\{123,321,231,312\}
\]
and every $n\geq0$,
\[
f_{S\cup\{132\}}(n,x)
=
f_{S\cup\{213\}}(n,x).
\]
Hence, $S\cup\{132\}$ and $S\cup\{213\}$ are DD-Wilf equivalent.
\end{cor}

\begin{proof}
Since $\Phi$ fixes every pattern in $S$ and interchanges $132$ and $213$,
\(
\Phi\bigl(S\cup\{132\}\bigr)=S\cup\{213\}.
\)
It therefore gives a bijection
\(
\Phi:
\Av_n(S\cup\{132\})
\longrightarrow
\Av_n(S\cup\{213\}).
\)
Moreover, the preceding proposition shows that
\(
\operatorname{DD}(\Phi(\pi))=\operatorname{DD}(\pi).
\)
Thus, the bijection preserves the exponent in the DD generating function, proving
\(
f_{S\cup\{132\}}(n,x)
=
f_{S\cup\{213\}}(n,x).
\)
\end{proof}

In particular, taking $S=\varnothing$ yields
\(
f_{132}(n,x)=f_{213}(n,x),
\)
so the patterns $132$ and $213$, whose avoidance classes were studied in Sections~\ref{sec:132} and~\ref{sec:213}, are DD-Wilf equivalent. 
This grouping differs from the fixed-point classification recalled in Section~\ref{sec:intro}.


In the next section, we apply the corollary to avoidance classes
defined by two forbidden patterns of length $3$.


\section{Two-pattern avoidance}
\label{sec:2pattern}

Since
\[
\Av_n(\sigma,\tau)
=
\Av_n(\sigma)\cap\Av_n(\tau),
\]
we refine the preceding one-pattern decompositions by imposing the second
avoidance condition on their blocks.

The Erd\H{o}s--Szekeres theorem \cite{ErdosSzekeres1935}
implies that every permutation of length at least \(5\) contains either an
increasing or a decreasing subsequence of length \(3\).
Consequently,
\(
\Av_n(123,321)=\varnothing, n\ge5.
\)
Hence,
\[
f_{\{123,321\}}(n,x)=0,
\qquad n\ge5,
\]
so this class is finite.

By Proposition~\ref{prop:312}, every double deficiency implies an
occurrence of the pattern $312$. Hence, if
\(
312\in S,
\)
then every permutation in $\Av_n(S)$ has no double deficiencies, and
therefore
\[
f_S(n,x)=|\Av_n(S)|.
\]
By Section~\ref{sec:DD-Wilf}, replacing
$132$ by $213$ (or vice versa) does not change the double-deficiency
generating function when the other forbidden pattern is one of
$123,321,231$, or $312$. Hence, after excluding the trivial class
$\Av(123,321)$ and the five classes containing $312$, it suffices to
consider the six classes
\[
\{321,231\},\quad
\{321,132\},\quad
\{132,231\},\quad
\{132,123\},\quad
\{132,213\},\quad
\{123,231\}.
\]


\subsection{The class $\Av(321,231)$}

Let \(S=\{321,231\}\). In the \(231\)-decomposition
\(\pi=L\,n\,R\), with \(\pi(i)=n\), the additional avoidance of
\(321\) forces \(R=[i,i+1,\ldots,n-1]\), while
\(L\in\Av_{i-1}(S)\). Hence,
\[
\pi=[L,n,i,i+1,\ldots,n-1].
\]

For \(i<n\), precisely \(j=i+1,\ldots,n-1\) are new double
deficiencies, since \(\pi(j)=j-1\) and \(\pi^{-1}(j)=j+1\); thus the
contribution is \(x^{n-i-1}f_S(i-1,x)\). For \(i=n\), it is
\(f_S(n-1,x)\).

Summing over all possible positions of \(n\), we obtain
\[
f_S(n,x)
=
f_S(n-1,x)
+
\sum_{i=1}^{n-1}
x^{n-i-1}f_S(i-1,x).
\]
The corresponding C-finite recurrence is, for $n\ge0$,
\[
f_S(n+2,x)
=
(x+1)f_S(n+1,x)
-
(x-1)f_S(n,x),
\]
where \(f_S(0,x)=1\), \(f_S(1,x)=1\), and \(f_S(2,x)=2\).

The mean is \((n-2)/4\) for \(n\ge2\), and
Figure~\ref{fig:321231} shows the distribution at \(n=170\).

\begin{figure}
\centering
\includegraphics[width=0.75\textwidth]{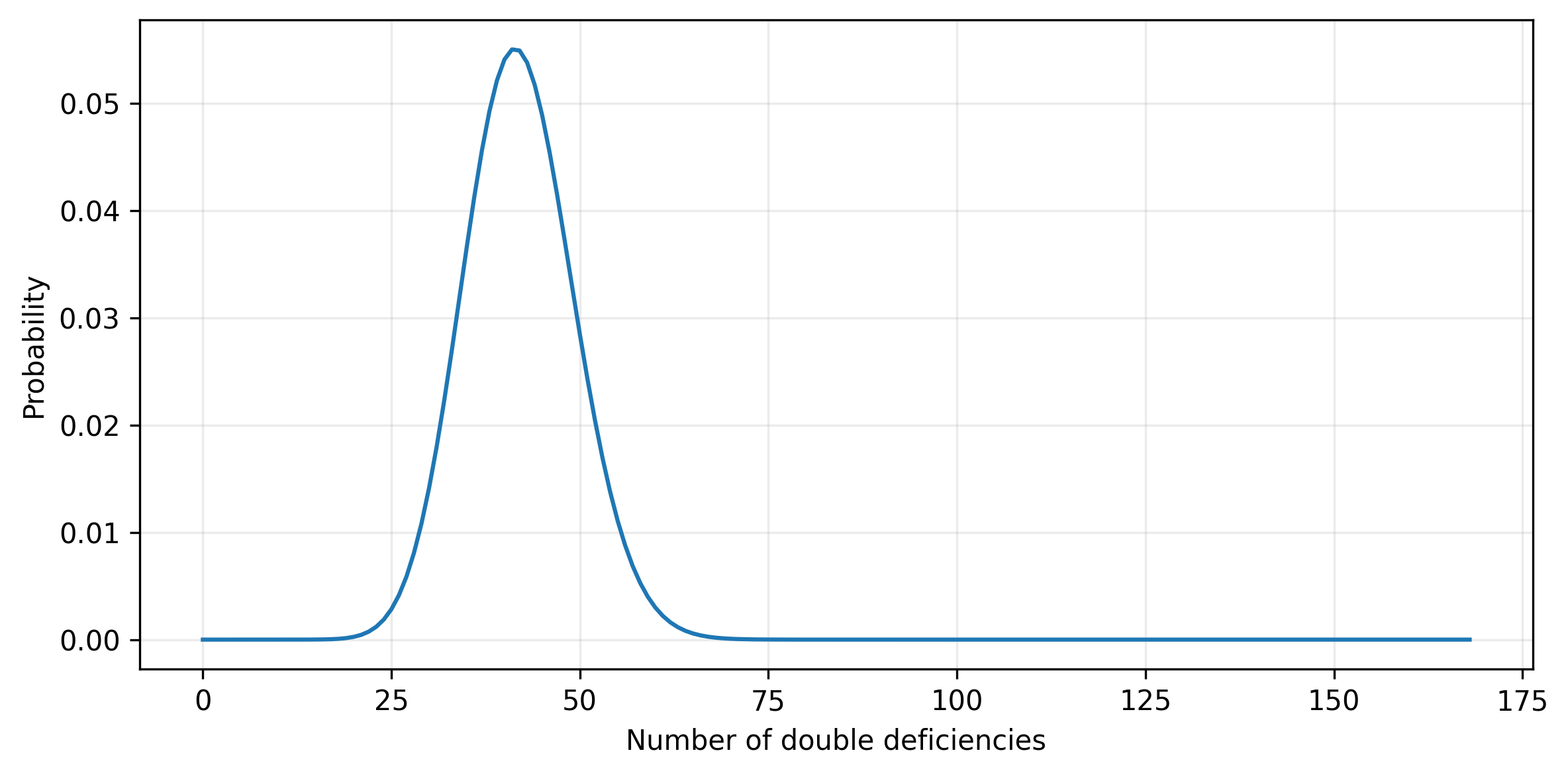}
\caption{Probability distribution of the number of double deficiencies
in \(\Av_{170}(321,231)\).}
\label{fig:321231}
\end{figure}


\subsection{The class $\Av(321,132)$}

Let \(S=\{321,132\}\) and use the \(132\)-decomposition
\(\pi=L\,n\,R\), where \(\pi(i)=n\). If \(i<n\), avoidance of \(321\)
forces both blocks to be increasing: a descent in \(R\), preceded by
\(n\), forms a \(321\)-pattern, while a descent in \(L\), followed by any
entry of the nonempty block \(R\), does the same. Thus,
\[
\pi=
[n-i+1,n-i+2,\ldots,n-1,n,1,2,\ldots,n-i].
\]

For this permutation, positions before or at \(i\) fail the first
inequality. In the right block, \(\pi(j)=j-i\); moreover
\(\pi^{-1}(j)=i+j\) exactly when \(j\le n-i\). Hence, the double
deficiencies are \(j=i+1,\ldots,n-i\), numbering
\(\max\{0,n-2i\}\). If \(i=n\), deleting the final maximum leaves an
arbitrary member of \(\Av_{n-1}(S)\) and creates no new deficiency.
Therefore,
\[
f_S(n,x)
=
f_S(n-1,x)
+
\sum_{i=1}^{n-1}
x^{\max\{0,n-2i\}},
\]
with
\(
f_S(0,x)=1.
\)
In particular,
\[
f_S(n,1)
=
\binom{n}{2}+1.
\]
The corresponding C-finite recurrence is, for $n\ge0$,
\[
\begin{aligned}
f_S(n+5,x)
={}&(x+2)f_S(n+4,x)
-2x f_S(n+3,x)
-2f_S(n+2,x)\\
&+(2x+1)f_S(n+1,x)
-x f_S(n,x),
\end{aligned}
\]
where $f_S(0,x)=1, f_S(1,x)=1, f_S(2,x)=2, f_S(3,x)=x+3, f_S(4,x)=x^2+x+5$
and $f_S(5,x)=x^3+x^2+2x+7$.


For the remaining four classes, we use the shift parameter \(c\ge0\) and
write
\[
f_S(n,c,x)
=
\sum_{\pi\in\Av_n(S)}
x^{\operatorname{DD}_c(\pi)}.
\]
Thus,
\[
f_S(n,x)=f_S(n,0,x).
\]


\subsection{The class $\Av(132,231)$}

Let \(S=\{132,231\}\) and write \(\pi=L\,n\,R\). The two avoidance
conditions require every entry of \(L\) to be, respectively, larger and
smaller than every entry of \(R\). Hence, \(L\) and \(R\) cannot both be
nonempty.
Thus, for \(n\ge2\),
\[
\pi=nR
\qquad\text{or}\qquad
\pi=L\,n,
\]
with the nonempty block in \(\Av_{n-1}(S)\). The first form contributes
\(f_S(n-1,c+1,x)\), with no extra contribution by
Lemma~\ref{lem:231} with \(i=1\). In the second, the same lemma with \(i=n\)
allows one additional shifted deficiency precisely when
\[
1\le j\le n-1
\qquad\text{and}\qquad
j+c=n.
\]
Thus this can occur only at \(j=n-c\). Define
\[
\triangle=
\begin{cases}
1,&n\in\{c+1,c+2,\ldots,c+n-1\},\\[1mm]
0,&\text{otherwise}.
\end{cases}
\]
When \(\triangle=1\), this candidate does satisfy both inequalities:
since \(j=n-c\in L\), \(c\ge1\), and \(\pi^{-1}(n)=n\),
\[
\pi(j)<j+c=n<\pi^{-1}(n)+c=n+c.
\]

The two forms therefore give, for \(n\ge2\) and \(n>c\),
\[
f_S(n,c,x)
=
f_S(n-1,c+1,x)
+
x^{\triangle}f_S(n-1,c,x).
\]

Together with \(f_S(0,c,x)=f_S(1,c,x)=1\), the boundary values are
\[
f_S(n,c,x)=2^{n-1},
\qquad 1\le n\le c.
\]
Here \(|\Av_n(S)|=2^{n-1}\), and no shifted value is admissible when
\(n\le c\). In particular, \(f_S(n,c,1)=2^{n-1}\) for \(n\ge1\).

At \(c=0\), the resulting C-finite recurrence is, for \(n\ge1\),
\[
f_S(n+3,x)
=
(x+1)f_S(n+2,x)
-
(x-2)f_S(n+1,x)
-
2f_S(n,x),
\]
where $f_S(0,x)=1, f_S(1,x)=1, f_S(2,x)=2$ and $f_S(3,x)=x+3$.

\begin{figure}
\centering
\includegraphics[width=0.75\textwidth]{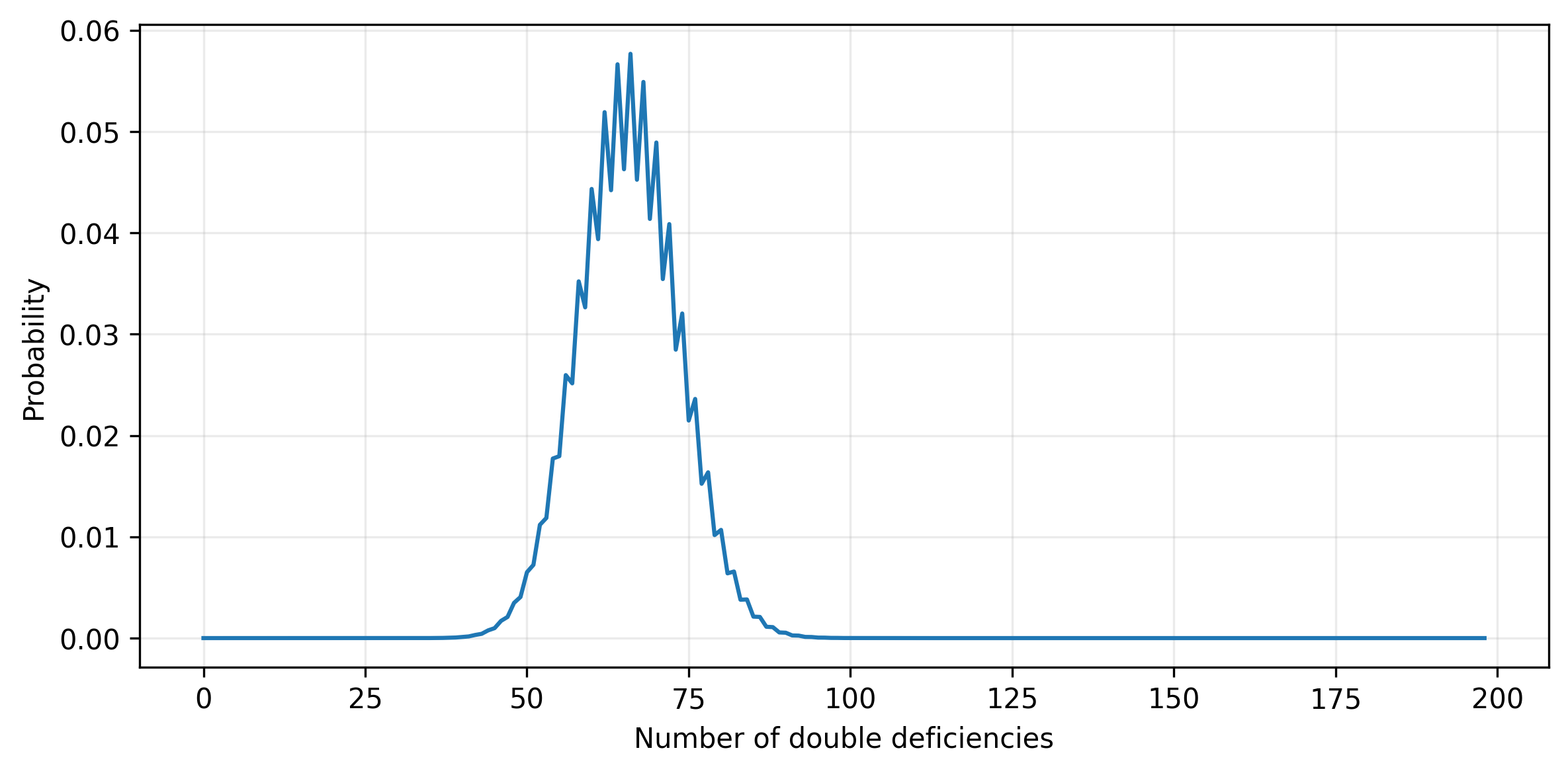}
\caption{Probability distribution of the number of double deficiencies
in \(\Av_{200}(132,231)\).}
\label{fig:132231}
\end{figure}

Figure~\ref{fig:132231} shows the distribution at \(n=200\).

\subsection{The class $\Av(132,123)$}

Let \(S=\{132,123\}\). In the \(132\)-decomposition
\(\pi=L\,n\,R\), with \(\pi(i)=n\), avoidance of \(123\) forces
\[
L=[n-1,n-2,\ldots,n-i+1].
\]
Indeed, an increasing pair in \(L\) would complete a \(123\)-pattern with
\(n\); meanwhile \(R\in\Av_{n-i}(S)\). Thus
\[
\pi=
[n-1,n-2,\ldots,n-i+1,n,R].
\]

The suffix contributes \(f_S(n-i,c+i,x)\). The decreasing block has no
internal contribution: if \(j+c\) is also in \(L\), then
\(\pi^{-1}(j+c)+c=\pi(j)\). As in the $\Av(132)$  case, the maximum adds one
shifted deficiency exactly when \(j=n-c\) lies in \(L\), i.e., when
\(c+1\le n\le c+i-1\). Define
\[
\triangle_i=
\begin{cases}
1,
& n\in\{c+1,c+2,\ldots,c+i-1\},\\[1mm]
0,
& \text{otherwise}.
\end{cases}
\]
Then, for \(n\ge1\) and \(0\le c<n\),
\[
f_S(n,c,x)
=
\sum_{i=1}^{n}
x^{\triangle_i}f_S(n-i,c+i,x).
\]
Together with \(f_S(0,c,x)=f_S(1,c,x)=1\), the boundary values are
\[
f_S(n,c,x)=|\operatorname{Av}_n(132,123)|=2^{n-1},
\qquad 1\le n\le c.
\]
Also \(f_S(n,c,1)=2^{n-1}\) for \(n\ge1\). At \(c=0\), the corresponding
C-finite recurrence is, for \(n\ge1\),
\[
f_S(n+3,x)
=
f_S(n+2,x)
+
4f_S(n+1,x)
-
4f_S(n,x),
\]
where $f_S(0,x)=1, f_S(1,x)=1, f_S(2,x)=2$ and $f_S(3,x)=x+3$.

\subsection{The class $\Av(132,213)$}

Let \(S=\{132,213\}\). In the \(132\)-decomposition
\(\pi=L\,n\,R\), avoidance of \(213\) forces \(L\) to be increasing
(a descent there completes a \(213\)-pattern with \(n\)). Hence,
\[
\pi=
[n-i+1,n-i+2,\ldots,n-1,n,R].
\]
Here \(R\in\Av_{n-i}(S)\), contributing \(f_S(n-i,c+i,x)\). There are
no cross-block contributions: a position in the prefix paired with a
shifted value in \(R\) fails the first inequality, while a position in
\(R\) paired with a shifted value in the prefix fails the second. Within
the fixed prefix,
\(\pi(j)=n-i+j\), and for an admissible shifted value in that prefix,
\(\pi^{-1}(j+c)=j+c-(n-i)\). Both strict inequalities therefore hold
exactly when \(n-i<c\). The position \(i\) itself cannot contribute, so
in this case the contributing positions are
\(j=1,\ldots,\min\{i-1,n-c\}\). Since
\(n-i<c\) is equivalent to \(n-c\le i-1\), 
\(\min\{i-1,n-c\}=n-c\). Therefore,
\[
\triangle_i=
\begin{cases}
n-c,&i>n-c,\\[1mm]
0,&i\le n-c.
\end{cases}
\]

Thus, for \(n\ge1\) and \(0\le c<n\),
\[
f_S(n,c,x)
=
\sum_{i=1}^{n}
x^{\triangle_i}f_S(n-i,c+i,x).
\]
Together with \(f_S(0,c,x)=f_S(1,c,x)=1\), the boundary values are
\[
f_S(n,c,x)=|\operatorname{Av}_n(132,213)|=2^{n-1},
\qquad 1\le n\le c.
\]
Also \(f_S(n,c,1)=2^{n-1}\) for \(n\ge1\). At \(c=0\), the corresponding
C-finite recurrence is, for \(n\ge1\),
\[
f_S(n+4,x)
=
(x+1)f_S(n+3,x)
-
(x-4)f_S(n+2,x)
-
4(x+1)f_S(n+1,x)
+
4x\,f_S(n,x),
\]
where $f_S(0,x)=1, f_S(1,x)=1, f_S(2,x)=2, f_S(3,x)=x+3$
and $f_S(4,x) = x^2+7$.


\subsection{The class $\Av(123,231)$}

Let \(S=\{123,231\}\) and use the \(231\)-decomposition with maximum in
position \(i\). For \(0\le c<n\), the case \(i=1\) contributes
\(f_S(n-1,c+1,x)\). If \(i>1\), avoidance of \(123\) forces both blocks
to be decreasing: an ascent in \(L\), together with \(n\), gives a
\(123\)-pattern, while \(1\in L\), so an ascent in \(R\), together with
that \(1\), also gives a \(123\)-pattern. Thus the permutation is
\[
\pi=
[i-1,i-2,\ldots,1,n,n-1,n-2,\ldots,i].
\]

The blocks have no internal shifted deficiencies: within either
consecutive decreasing block, \(\pi^{-1}(j+c)+c=\pi(j)\), so the two
strict inequalities cannot both hold. By
Lemma~\ref{lem:231}, the remaining positions form
\([1,i-1]\cap[i-c,n-c]\), so the number of them is
\[
\triangle_i
=
\min\{c,i-1,n-c,n-i+1\}.
\]

Thus, for \(n\ge1\) and \(0\le c<n\),
\[
f_S(n,c,x)
=
f_S(n-1,c+1,x)
+
\sum_{i=2}^{n}
x^{\min\{c,i-1,n-c,n-i+1\}}.
\]

Together with \(f_S(0,c,x)=f_S(1,c,x)=1\), the boundary values are
\[
f_S(n,c,x)=\binom{n}{2}+1,
\qquad n\le c.
\]
Here \(|\Av_n(S)|=\binom n2+1\). Also
\(f_S(n,c,1)=\binom n2+1\). At \(c=0\), the corresponding C-finite
recurrence is, for \(n\ge0\),
\[
\begin{aligned}
f_S(n+8,x)
={}&f_S(n+7,x)+2f_S(n+6,x)+(x-2)f_S(n+5,x)\\
&-(x+1)f_S(n+4,x)-(2x-1)f_S(n+3,x)\\
&+2x\,f_S(n+2,x)+x\,f_S(n+1,x)-x\,f_S(n,x),
\end{aligned}
\]
where $f_S(0,x)=1, f_S(1,x)=1, f_S(2,x)=2, f_S(3,x)=x+3, f_S(4,x)=2x+5, f_S(5,x)=5x+6,  
f_S(6,x)=x^2+6x+9, f_S(7,x)=2x^2+10x+10$ and $f_S(8,x)=5x^2+10x+14$.

\section{Conclusion}

The results show that the behavior of double deficiencies varies substantially
across length-$3$ avoidance classes, even when the underlying classes have the
same ordinary enumeration. 
For $\Av(321)$ the bivariate generating function is algebraic and the distribution admits an explicit binomial approximation; for $\Av(312)$ the
statistic vanishes identically; and for $\Av(132)$, $\Av(213)$ and $\Av(231)$
the shifted double-deficiency statistic yields convolution-type recurrences
that make the distributions computationally accessible. 
Together with the simplifications of Section~\ref{sec:2pattern}, the
DD-preserving symmetry
of Section~\ref{sec:DD-Wilf} reduces the two-pattern analysis to six cases, each of which
satisfies a C-finite recurrence.

In the $321$-avoiding case, Proposition~\ref{prop:binomial} gives exact
formulas for the mean and the variance and shows that
$\dTV(\Law(X_n),\Bin(n-2,\tfrac14))=O(n^{-1/2}(\log n)^{1/4})$.
The logarithmic factor arises from the way the total-variation sum is split, and
numerical evidence suggests that the true order is $n^{-1/2}$; determining the
exact order, and the corresponding constant, would sharpen the result.

Several questions remain. A comparable recurrence for the double-deficiency
generating functions of $123$-avoiding permutations has not been obtained
(Open Problem~\ref{op:123}). The asymptotic behavior of the distributions for the remaining
single-pattern classes is also open: 
the data suggest that the distribution for $\Av(231)$
is nearly symmetric, whereas the one for $\Av(132)$ is strongly skewed, with a
mean that grows sublinearly in $n$ rather than linearly as in $\Av(321)$ and
$\Av(231)$. Finally, it would be natural to investigate double deficiencies in
avoidance classes defined by patterns of length $4$ and, more generally,
longer patterns.

\appendix

\section{Proof of Proposition~\ref{prop:binomial}}\label{app:prop1}

\emph{Support.} Every red level step lies above the \(x\)-axis, so the path
contains at least one up-step before it and at least one down-step after it;
thus \(X_n\le n-2\).
Conversely, the path consisting of an up-step, \(n-2\) level steps, and a
down-step has all of its level steps above the \(x\)-axis, and coloring
exactly \(j\) of them red realizes every \(j\in\{0,1,\dots,n-2\}\).

\emph{Moments.} For \(n=2\), we have \(X_2=0\), so the stated formulas hold
immediately. We may therefore assume \(n\ge3\). Write
\(S=\sqrt{1-4z}\) and \(C=C(z)=(1-S)/(2z)\), so that
\(C=1+zC^{2}\), \(S=1-2zC\), and, by Remark~\ref{rem:catalan},
\(F_1(z,1)=C\) and \(F_0(z,1)=C^{2}\). Taking the first and
second derivatives of \eqref{eq:F0-func} and \eqref{eq:F1-func} with respect
to \(x\), evaluating at \(x=1\), and using \(z^{2}C^{2}=z(C-1)\), which gives the
coefficients \(1-2z-2z^{2}C^{2}=S\) and
\(1-z-z^{2}C^{2}=1/C\), yields
\[
\begin{gathered}
\frac{\partial F_0}{\partial x}\bigg|_{x=1}
=\frac{zC^{2}}{S},
\qquad
\frac{\partial F_1}{\partial x}\bigg|_{x=1}
=\frac{z^{3}C^{4}}{S},
\qquad
\frac{\partial^{2}F_0}{\partial x^{2}}\bigg|_{x=1}
=
\frac{2z^{2}C^{2}}{S^{2}}
+\frac{2z^{4}C^{4}}{S^{3}},
\\[4pt]
\frac{\partial^{2}F_1}{\partial x^{2}}\bigg|_{x=1}
=
\frac{2z^{4}C^{4}}{S^{2}}
+\frac{2z^{6}C^{6}}{S^{3}}
+\frac{2z^{6}C^{7}}{S^{2}}.
\end{gathered}
\]
Since
\[
[z^m]C^{k}(1-4z)^{-1/2}=\binom{2m+k}{m}
\]
\cite[eq.~(5.72)]{GKP}, the formula for
\(\partial F_1/\partial x\big|_{x=1}\) gives
\[
\mathbb E[X_n]
=
\frac{1}{C_n}[z^n]\frac{z^{3}C^{4}}{S}
=
\frac{1}{C_n}\binom{2n-2}{n-3}
=
\frac{(n-1)(n-2)}{2(2n-1)}.
\]
Substituting \(C=(1-S)/(2z)\) in the formula for
\(\partial^{2}F_1/\partial x^{2}\big|_{x=1}\) and reducing every even power
of \(S\) using \(S^{2}=1-4z\) gives
\[
\frac{\partial^{2}F_1}{\partial x^{2}}\bigg|_{x=1}
=
\frac{9-26z+26z^{2}-6z^{3}-z^{-1}}{(1-4z)^{3/2}}
+\Bigl(\frac1z-3+2z\Bigr).
\]
The parenthesized part does not contribute to \([z^n]\) for \(n\ge2\).
Since
\[
[z^m](1-4z)^{-3/2}
=
(2m+1)\binom{2m}{m},
\]
expressing the resulting central binomial coefficients as multiples of
\(\binom{2n}{n}\) yields
\[
\mathbb E\bigl[(X_n)_2\bigr]
=
\frac{(n-2)(n-3)(n^{2}-3n+6)}
{4(2n-1)(2n-3)},
\qquad
(u)_2=u(u-1).
\]
The stated variance follows from
\[
\Var(X_n)
=
\mathbb E[(X_n)_2]
+\mathbb E[X_n]
-\mathbb E[X_n]^{2}.
\]

\emph{Singular expansion.} Write
\(
p_n(x)=[z^n]F_1(z,x),
\)
so that \(p_n(x)/p_n(1)\) is the probability generating function of
\(X_n\) and \(p_n(1)=C_n\). The discriminant in the expression for
\(F_1\) factors as
\[
1-2z-2xz-3z^{2}+2xz^{2}+x^{2}z^{2}
=
\bigl(1-(x+3)z\bigr)\bigl(1-(x-1)z\bigr)
=:D(z,x),
\]
and rationalizing gives
\[
F_1(z,x)
=
\frac{\bigl(1+(x-1)z\bigr)-\sqrt{D(z,x)}}
{2z\bigl(x-(x-1)z\bigr)}.
\]
Hence, the only possible singularities of \(z\mapsto F_1(z,x)\) are
the branch points \((x+3)^{-1}\) and \((x-1)^{-1}\), together with the
possible pole \(x/(x-1)\); the apparent singularity at \(z=0\) is
removable. Fix \(\varepsilon>0\) sufficiently small and put
\[
\mathcal N
=
\{x\in\mathbb C:|x-1|\le\varepsilon\}.
\]
For \(x\in\mathcal N\), the last two possible singularities have modulus
at least \(1\). Therefore,
\[
\rho(x)=\frac{1}{x+3}
\]
is the unique dominant singularity, and \(F_1(\,\cdot\,,x)\) is
\(\Delta\)-analytic there, uniformly for \(x\in\mathcal N\).

Put \(\eta=1-(x+3)z\). The rationalized form gives
\[
F_1(z,x)
=
\mathcal R(z,x)
-
\frac{\sqrt{1-(x-1)z}}
{2z\bigl(x-(x-1)z\bigr)}\eta^{1/2},
\]
where
\[
\mathcal R(z,x)
=
\frac{1+(x-1)z}
{2z\bigl(x-(x-1)z\bigr)}.
\]
The factor multiplying \(\eta^{1/2}\) is analytic at \(z=\rho(x)\),
uniformly for \(x\in\mathcal N\), and its value there is
\[
\frac{\sqrt{1-(x-1)\rho(x)}}
{2\rho(x)\bigl(x-(x-1)\rho(x)\bigr)}
=
\frac{(x+3)^{3/2}}{(x+1)^{2}}.
\]
Its Taylor expansion at \(z=\rho(x)\) therefore gives
\[
F_1(z,x)
=
\mathcal R(z,x)
-
\frac{(x+3)^{3/2}}{(x+1)^{2}}\eta^{1/2}
+
O\bigl(\eta^{3/2}\bigr),
\]
uniformly for \(x\in\mathcal N\).

Although \(\mathcal R\) has a simple pole at \(z=0\), its principal part is
\(1/(2xz)\), which does not contribute to \([z^n]\) for \(n\ge0\). After
subtracting this principal part, \(\mathcal R\) is
analytic in a disk whose radius is uniformly larger than
\(|\rho(x)|\). Hence, its coefficients of nonnegative powers are
exponentially smaller than \((x+3)^n\). Applying uniform singularity
analysis
\cite[Chapter~VI and Section~IX.7]{FlajoletSedgewick} to \(F_1-\mathcal R\)
and adding back the contribution of \(\mathcal R\) therefore gives
\[
p_n(x)
=
\frac{(x+3)^{3/2}}{2\sqrt{\pi}\,(x+1)^{2}}
(x+3)^{n}n^{-3/2}
\bigl(1+O(n^{-1})\bigr),
\]
uniformly for \(x\in\mathcal N\). At \(x=1\), this reduces to
\(
C_n
\sim
\frac{4^n}{\sqrt{\pi}\,n^{3/2}}.
\)
Dividing by \(p_n(1)\), we obtain
\[
\frac{p_n(x)}{p_n(1)}
=
A(x)B(x)^{n}\bigl(1+O(n^{-1})\bigr),
\]
where
\[
A(x)=\frac{(x+3)^{3/2}}{2(x+1)^{2}},
\qquad
B(x)=\frac{x+3}{4}.
\]

\emph{Characteristic functions.} Let
\[
\phi_n(t)=\frac{p_n(e^{it})}{p_n(1)}
\qquad\text{and}\qquad
\psi_n(t)=B(e^{it})^{\,n-2}.
\]
The latter is the characteristic function of
\(\Bin(n-2,\frac14)\), whose probability generating function is
\(B(x)^{n-2}\). Put \(h=AB^{2}\). Then \(h\) is analytic near \(x=1\)
and
\[
h(1)=A(1)B(1)^{2}=1.
\]
Therefore,
\(
\phi_n(t)
=
\psi_n(t)\bigl(h(e^{it})+O(n^{-1})\bigr)
\)
uniformly for \(|t|\le\delta\), where \(\delta>0\) is sufficiently small
that \(e^{it}\in\mathcal N\) whenever \(|t|\le\delta\). Since
\(
|h(e^{it})-1|=O(|t|)
\)
and
\[
\bigl|B(e^{it})\bigr|^{2}
=
\frac{10+6\cos t}{16}
=
1-\frac34\sin^{2}\frac t2
\le
e^{-3t^{2}/(4\pi^{2})},
\]
we obtain
\[
|\phi_n(t)-\psi_n(t)|
\le
K\bigl(|t|+n^{-1}\bigr)e^{-cnt^{2}},
\qquad
|t|\le\delta,
\]
where \(K,c>0\) are independent of \(n\) and \(t\).

For \(\delta\le |t|\le\pi\) and \(x=e^{it}\), the three possible
singularities listed above have moduli
\[
|x+3|^{-1}\ge\frac14+\gamma,
\qquad
|x-1|^{-1}\ge\frac12,
\qquad
\left|\frac{x}{x-1}\right|\ge\frac12
\]
for some \(\gamma=\gamma(\delta)>0\). Hence,
\(F_1(z,e^{it})\) is analytic and uniformly bounded on
\(|z|\le R\) for some \(R>\frac14\). Cauchy's estimate, together with
\[
p_n(1)=C_n
\sim
\frac{4^n}{\sqrt{\pi}\,n^{3/2}},
\]
shows that \(\phi_n(t)\) is uniformly exponentially small on this
range. The same holds for \(\psi_n(t)\), since
\(|B(e^{it})|<1\) uniformly there. Parseval's identity applied to
\(\phi_n-\psi_n\) now gives
\[
\sum_{j\in\mathbb Z}
\left(
\mathbb P(X_n=j)
-
\mathbb P\bigl(\Bin(n-2,\tfrac14)=j\bigr)
\right)^{2}
=
O\bigl(n^{-3/2}\bigr).
\]

\emph{Conclusion.} For real \(s\) sufficiently close to \(0\), the same
uniform estimate gives
\[
\mathbb E[e^{sX_n}]
\le
K\exp\bigl(n\Lambda(s)\bigr),
\qquad
\Lambda(s)=\log B(e^s).
\]
Moreover,
\[
\Lambda'(0)=\frac14,
\qquad
\Lambda''(0)=\frac3{16},
\]
and hence
\[
\Lambda(s)
=
\frac{s}{4}+\frac{3s^{2}}{32}+O(s^{3}).
\]
The same bound holds for \(\Bin(n-2,\frac14)\), whose moment generating
function is exactly
\(
e^{(n-2)\Lambda(s)}.
\)
The Cram\'er--Chernoff method
\cite[Section~2.2]{BLM}, applied with
\(
s=2\sqrt{\frac{\log n}{n}}
\)
for the upper tail and its negative for the lower tail, therefore shows
that both distributions put mass \(O(n^{-1})\) outside
\[
I_n
=
\left\{
j:
\left|j-\frac n4\right|
\le
\sqrt{n\log n}
\right\}.
\]
Since \(\#I_n=O(\sqrt{n\log n})\), splitting the total-variation sum at
\(I_n\) and applying the Cauchy--Schwarz inequality inside \(I_n\) gives
\[
\begin{aligned}
2\,\dTV\Bigl(
\Law(X_n),
\Bin\bigl(n-2,\tfrac14\bigr)
\Bigr)
&\le
(\#I_n)^{1/2}
\left(
\sum_{j\in\mathbb Z}
\left(
\mathbb P(X_n=j)
-
\mathbb P\bigl(\Bin(n-2,\tfrac14)=j\bigr)
\right)^2
\right)^{1/2}
\\
&\qquad{}+O(n^{-1})
\\
&=
(\#I_n)^{1/2}O(n^{-3/4})+O(n^{-1})
\\
&=
O\bigl(n^{-1/2}(\log n)^{1/4}\bigr),
\end{aligned}
\]
as required. \hfill$\square$



\begin{thebibliography}{99}

\bibitem{BJS}
Sara C. Billey, William Jockusch, and Richard P. Stanley,
{\it Some combinatorial properties of Schubert polynomials},
Journal of Algebraic Combinatorics, Vol. 2 (1993), No. 4, pp. 345--374.

\bibitem{BLM}
St\'ephane Boucheron, G\'abor Lugosi, and Pascal Massart,
{\it Concentration Inequalities: A Nonasymptotic Theory of Independence},
Oxford University Press, Oxford, 2013.



\bibitem{CarlitzScoville}
L. Carlitz and Richard Scoville,
{\it Generalized Eulerian numbers: combinatorial applications},
Journal f\"ur die reine und angewandte Mathematik, Vol. 265 (1974),
pp. 110--137.

\bibitem{DokosEtAl2012}
Theodore Dokos, Tim Dwyer, Bryan P. Johnson, Bruce E. Sagan, and Kimberly Selsor,
{\it Permutation patterns and statistics},
Discrete Mathematics, Vol. 312 (2012), No. 18, pp. 2760--2775.

\bibitem{Elizalde2004}
Sergi Elizalde,
{\it Multiple pattern avoidance with respect to fixed points and excedances},
Electronic Journal of Combinatorics, Vol. 11 (2004), No. 1,
Research Paper 51, 40 pp.

\bibitem{ElizaldeCF}
Sergi Elizalde,
{\it Continued fractions for permutation statistics},
Discrete Mathematics \& Theoretical Computer Science,
Vol. 19 (2018), No. 2, Article \#11.

\bibitem{ElizaldeNoy}
Sergi Elizalde and Marc Noy,
{\it Consecutive patterns in permutations},
Advances in Applied Mathematics, Vol. 30 (2003), No. 1--2,
pp. 110--125.

\bibitem{ElizaldePak2004}
Sergi Elizalde and Igor Pak,
{\it Bijections for refined restricted permutations},
Journal of Combinatorial Theory, Series A, Vol. 105 (2004),
No. 2, pp. 207--219.

\bibitem{ErdosSzekeres1935}
Paul Erd\H{o}s and George Szekeres,
{\it A combinatorial problem in geometry},
Compositio Mathematica, Vol. 2 (1935), pp. 463--470.


\bibitem{FlajoletSedgewick}
Philippe Flajolet and Robert Sedgewick,
\emph{Analytic Combinatorics},
Cambridge University Press, Cambridge, 2009.



\bibitem{FoataZeilberger1990}
Dominique Foata and Doron Zeilberger,
{\it Denert's permutation statistic is indeed Euler--Mahonian},
Studies in Applied Mathematics, Vol. 83 (1990), No. 1, pp. 31--59.




\bibitem{Fu}
Amy M. Fu,
{\it A context-free grammar for peaks and double descents of permutations},
Advances in Applied Mathematics, Vol. 100 (2018), pp. 179--196.




\bibitem{GKP}
Ronald L. Graham, Donald E. Knuth, and Oren Patashnik,
{\it Concrete Mathematics}, 2nd ed., Addison--Wesley, Reading, MA, 1994.



\bibitem{MAZ}
Tipaluck Krityakierne, Thotsaporn ``Aek'' Thanatipanonda, and Doron Zeilberger,
{\it A quick proof that 321-avoiding permutations without double deficiencies
are counted by the Motzkin numbers}, arXiv:2607.05431 (2026).


\bibitem{RSZ}
Aaron Robertson, Dan Saracino, and Doron Zeilberger,
{\it Refined restricted permutations},
Annals of Combinatorics, Vol. 6 (2002), No. 3, pp. 427--444.


\bibitem{RS}
Martin Rubey and Christian Stump,
{\it Double deficiencies of Dyck paths via the Billey--Jockusch--Stanley bijection},
Journal of Integer Sequences, Vol. 20 (2017), Article 17.9.6.

\bibitem{SimionSchmidt1985}
Rodica Simion and Frank W. Schmidt,
{\it Restricted permutations},
European Journal of Combinatorics, Vol. 6 (1985), No. 4, pp. 383--406.


\bibitem{Vatter2002}
Vincent R. Vatter,
{\it Permutations avoiding two patterns of length three},
Electronic Journal of Combinatorics, Vol. 9 (2003), No. 2, Article \#R6.

\end{thebibliography}
\end{document}